\documentclass{article}
\usepackage{amsthm}
\usepackage{amsmath}
\usepackage{amsfonts}

\usepackage{mathtools}
\mathtoolsset{showonlyrefs}

\theoremstyle{plain}
\newtheorem{thm}{Theorem}[section]
\newtheorem{lem}[thm]{Lemma}
\newtheorem{prop}[thm]{Proposition}

\theoremstyle{definition}
\newtheorem{dfn}[thm]{Definition}
\newtheorem{ntn}[thm]{Notation}

\newcommand{\diam}{{\rm diam}}
\newcommand{\interior}{{\rm Int}}
\newcommand{\Con}{{\rm Con}}
\newcommand{\tifs}{\{\phi_\tau\}_{\tau\in T\setminus\{\emptyset\}}}

\makeatletter
\def\blfootnote{\gdef\@thefnmark{}\@footnotetext}
\makeatother

\title{Hausdorff dimension of the limit sets of\\
Tree Iterated Function Systems}
\author{Hiromichi ONO}
\date{February 28, 2026}

\begin{document}

\maketitle

\begin{abstract}
We investigate Tree Iterated Function Systems (TIFSs), which we introduce in this paper. TIFSs are the generalizations of Iterated Function Systems in which we take the maps independently at each step and each block. In this paper, we give the definition of TIFSs and the limit sets of them. We show a formula for the Hausdorff dimension of the limit sets of TIFSs, which is a generalization of Bowen's formula. Moreover, we give an example which emphasizes the difference between TIFSs and non-autonomous IFSs.
\end{abstract}

\section{Introduction and the main results}

\blfootnote{2020 Mathematics Subject Classification. Primary 28A80; Secondary 37C45.

Keywords: fractal, Hausdorff dimension, Tree Iterated Function System, limit set.}

In this paper, we investigate Tree Iterated Function Systems (TIFSs), which we introduce here as the generalization of IFSs. Roughly speaking, TIFSs are the IFSs in which we take the maps independently at each step and each block. The limit sets of TIFSs can be much more complicated than those of usual IFSs, since the self-similarity does not hold any more.

The motivations of considering TIFSs are the following. First, there are many fractal-like objects in nature, such as coastlines or clouds. However, some of them do not have self-similarity. Therefore, it is natural to consider generalized version of IFSs to investigate what can be said about such fractal shapes.

The second motivation is the connection with random fractals. Random fractals are the fractals which are constructed by taking maps randomly at each step. Then, such random fractals can be seen naturally as the limit sets of some TIFSs.

The generalizations of IFSs in this direction are studied in \cite{MR3449231} and \cite{MR2431670}. However, these are far from general settings. TIFSs are the generalizations of these systems (with finite alphabets).

The definition of TIFSs is as follows. Let $T$ be a rooted tree which is finitely splitting and pruned (see section \ref{Sec:Def}). Let $X$ be a compact connected metric space. A TIFS is a collection $\{\phi_\tau\}_{\tau\in T\setminus\{\emptyset\}}$ of maps $\phi_\tau:X\to X$ which are uniformly contractive. For a TIFS $\Psi=\tifs$, we can define the limit set $J$ as the set of all points $x$ in $X$ which can be represented as
\begin{equation}
\{x\}=\bigcap_{n=1}^\infty \phi_{\tau_1}\circ\cdots\circ\phi_{\tau_n}(X)
\end{equation}
for some infinite path $(\emptyset,\tau_1,\tau_2,\ldots)$ in $T$. We simply write $\Phi_{\tau_n}$ for $\phi_{\tau_1}\circ\cdots\circ\phi_{\tau_n}$.

To obtain the dimension formula, we restrict ourselves to a certain class of TIFSs, which we call conformal TIFS (See Definition \ref{dfn:cTIFS}). Here, all maps $\phi_\tau$ are supposed to be conformal. For such systems and for any $t\ge0$, we define $Z_n^*(t)$ by
\begin{equation}
Z_{n}^*(t)=\min_A\left(\sum_{\omega\in A}\|\Phi_\omega'\|^t\right),
\end{equation}
where $A$ ranges over all maximal antichains (see Definition \ref{Def:ma}) in $T^{(\le n)}\setminus\{\emptyset\}=\{\tau\in T\mid 0<|\tau|\le n\}$ (see Notation \ref{Notation}), and $\|\Phi_\omega'\|=\sup_{x\in X}|\Phi_\omega'(x)|$. Here, $|\Phi_\omega'(x)|$ denotes the operator norm of the derivative of $\Phi_\omega$ at $x$. Let $Z_\infty^*(t)=\lim_{n\to\infty}Z_n^*(t)$. We set $\beta^*(\Psi):=\inf\{t\ge0\mid Z_\infty^*(t)=0\}$.

For each $\tau\in T$, let $I_\tau$ be the set of all successors of $\tau$.

We now present the main theorem of this paper.

\begin{thm}[Theorem \ref{Thm:Main}]
Let $\Psi=\tifs$ be a conformal TIFS and $J$ be the limit set. Let $t\ge0$. Then, we have the following.
\begin{itemize}
\item[{\em (i)}] If $Z_\infty^*(t)=0$, then $\dim_H(J)\le t$.
\item[{\em (ii)}] In addition to the assumptions of our theorem, suppose
\begin{equation}\label{Eq:main_cond1}
\lim_{|\tau|\to\infty}\frac1{|\tau|}\log\#I_\tau=0.
\end{equation}
Then, if $Z_\infty^*(t)>0$, we have that $\dim_H(J)\ge t$.
\end{itemize}
Especially, if the condition \eqref{Eq:main_cond1} in {\em (ii)} is satisfied, then $\dim_H(J)=\beta^*(\Psi)$.
\end{thm}

The above theorem is an extension of the dimension formula for non-autonomous IFSs (see \cite{MR3449231}). The main difference between the above theorem and the dimension formula for non-autonomous IFSs is that, for TIFSs, we need to consider all maximal antichains in $T^{(\le n)}\setminus\{\emptyset\}$, while it is enough to consider $T^{(n)}$ (see Notation \ref{Notation}) only for non-autonomous IFSs. More rigorously, for any conformal TIFS $\Psi=\tifs$, if we define
\begin{equation}
Z_n(t)=\sum_{\tau\in T^{(n)}} \|\Phi_\tau'\|^t\quad(t\ge0),
\end{equation}
and
\begin{equation}
\beta(\Psi)=\inf\{t\ge0\mid \liminf_{n\to\infty}Z_n(t)=0\},
\end{equation}
then $\beta(\Psi)$ is the real number where the sign of the lower pressure function changes from positive to negative, and we have $\dim_H(J)=\beta(\Psi)$ if $\Psi$ is a non-autonomous IFS (see \cite{MR3449231}). In general, we have $\beta^*(\Psi)\le \beta(\Psi)$ for any conformal TIFS $\Psi$. However, as the following theorem describes, it may happen that $\beta^*(\Psi)<\beta(\Psi)$.

\begin{thm}[see Theorems \ref{Eg:Dim}, \ref{Eg:Zn}]
There exists a conformal TIFS $\Psi=\tifs$ such that
\begin{equation}
\dim_H(J)=\beta^*(\Psi)=\frac{\log2}{\log3}<1=\beta(\Psi).
\end{equation}
\end{thm}

We give the counterexample in section \ref{Sec:Example}.

In section \ref{Sec:Def}, we give the definition of TIFSs, and show some basic properties of TIFSs. In section \ref{Sec:Measure}, we construct nice measures on $T^{(n)}$, which is crucial for proving the main theorem. In section \ref{Sec:Proof}, we prove the main theorem. In section \ref{Sec:Example}, we construct a counterexample which we mentioned above.

\section{Tree Iterated Function Systems}\label{Sec:Def}
In this section, we give the definition of Tree Iterated Function Systems.

We say a partially ordered set $(T,\preceq)$ is a rooted tree if $T$ has a minimum element $\emptyset$ named root and for any $\tau\in T$, $\{\tau'\in T\mid \tau'\preceq\tau\}$ is isomorphic to a partially ordered set $(\{0,\ldots,n\},\le)$ for some $n\ge0$. For $\tau\in T$, the height $|\tau|$ of $\tau$ is the unique number $n$ such that $\{\tau'\in T\mid \tau'\preceq\tau\}$ is isomorphic to $(\{0,\ldots,n\},\le)$.

For the rest of this paper, let $(T,\preceq)$ be a rooted tree.

\begin{dfn}\label{dfn:TIFS}
    Let $(X,d)$ be a compact connected metric space and let $(T,\preceq)$ be a rooted tree which is pruned and finitely splitting (i.e. the number of the successors of each vertex is non-zero and finite). We denote the root of $T$ by $\emptyset$. We say that $\{\phi_\tau:X\to X\}_{\tau\in T\setminus\{\emptyset\}}$ is a Tree Iterated Function System (for short, TIFS) if it satisfies the following condition:
    \begin{itemize}
    \item There exists $s\in (0,1)$ such that, for all $\tau\in T\setminus \{\emptyset\}$ and all $x,y\in X$,
    \[
    d(\phi_\tau(x),\phi_\tau(y))\le sd(x,y).
    \]
    \end{itemize}
\end{dfn}

\begin{ntn}\label{Notation}
We use the following notations.
\begin{itemize}
\item Let $\partial T$ be the set of all simple paths starting from $\emptyset$. For $\tau\in T$ and $\omega\in \partial T$, we write $\tau\preceq \omega$ if $\omega$ passes through $\tau$.
\item For each $\tau\in T$ and $0\le n\le |\tau|$, let $\tau|_n$ be the unique vertex which satisfies $\tau|_n\preceq \tau$ and $|\tau|_n|=n$. For each $\tau\in \partial T$ and $n\in \mathbb Z_{\ge0}$, let $\tau|_n$ be the unique vertex which satisfies $\tau|_n\preceq\tau$ and $|\tau|_n|=n$.
\item For each $\tau\in T$, we set $[\tau]:=\{\omega\in \partial T\mid \tau\preceq\omega\}$.
\item We write $\tau\prec\omega$ or $\omega\succ\tau$ when $\tau\preceq\omega$ and $\tau\ne \omega$. For each $\tau\in T$, we set $I_\tau:=\{\omega\in T\mid \omega\succ\tau,\, |\omega|=|\tau|+1\}$.
\item For each $n\in\mathbb Z_{\ge0}$, we set $T^{(n)}:=\{\tau\in T\mid |\tau|=n\}$, $T^{(\le n)}:=\{\tau\in T\mid |\tau|\le n\}$. For each $n\in \mathbb Z_{\ge0}$ and $\omega\in T$, we set $T_{\succ\omega}^{(\le n)}:=\{\tau\in T\mid |\tau|\le n,\,\tau\succ\omega\}$, $T_{\succeq\omega}^{(\le n)}:=\{\tau\in T\mid |\tau|\le n,\,\tau\succeq\omega\}$. Also, we write $[\omega]|_n=\{\tau\in T^{(n)}\mid \tau\succeq\omega\}$.
\item For each $\tau\in T$, let $\Phi_\tau:=\phi_{\tau_1}\circ\cdots\circ\phi_{\tau_n}$, where $(\emptyset,\tau_1,\ldots,\tau_n)$ is the unique simple path connecting $\emptyset$ and $\tau$. (If $\tau=\emptyset$, let $\Phi_\emptyset={\rm Id}_X$.)
\end{itemize}
\end{ntn}

For each $\tau\in \partial T$, $\{\Phi_{\tau|_n}(X)\}_{n=1}^\infty$ are decreasing compact sets, therefore $\bigcap_{n=1}^\infty\Phi_{\tau|_n}(X)$ is a non-empty compact set. By the uniform contractivity,
\[
\diam(\Phi_{\tau|_n}(X))\le s^{n}\diam(X).
\]
Therefore, $\bigcap_{n=1}^\infty\Phi_{\tau|_n}(X)$ is a one-point set. Hence, we can give the following definition.

\begin{dfn}
    Let $\{\phi_\tau\}_{\tau\in T\setminus \{\emptyset\}}$ be a TIFS. The coding map $\pi:\partial T\to X$ is defined as the map satisfying the following.
    \[
    \bigcap_{n=1}^\infty\Phi_{\tau|_n}(X)=\{\pi(\tau)\}.
    \]

    The limit set $J$ of $\{\phi_\tau\}_{\tau\in T\setminus\{\emptyset\}}$ is defined by $J=\pi(\partial T)$.
\end{dfn}

\begin{prop}
    Let $\{\phi_\tau\}_{\tau\in T\setminus \{\emptyset\}}$ be a TIFS. Then,
    \begin{equation}
        J=\bigcap_{n=1}^\infty\bigcup_{\tau\in T^{(n)}}\Phi_\tau(X).
    \end{equation}

    Especially, $J$ is compact.
\end{prop}
\begin{proof}
    If $x\in J$, then there exists $\omega \in \partial T$ such that $\pi(\omega)=x$. By definition of $\pi$, $x\in \Phi_{\omega|_n}(X)$ for any $n$. Therefore, $x\in\bigcap_{n=1}^\infty\bigcup_{\tau\in T^{(n)}}\Phi_\tau(X)$ and it follows that $J\subset \bigcap_{n=1}^\infty\bigcup_{\tau\in T^{(n)}}\Phi_\tau(X)$.

    To show the converse, let $x\in \bigcap_{n=1}^\infty\bigcup_{\tau\in T^{(n)}}\Phi_\tau(X)$. Let $T_x:=\{\tau\in T\mid x\in\Phi_\tau(X)\}$. Then, $T_x$ is finitely splitting, and has infinitely many vertices. By K\H{o}nig's lemma, there exists a simple path starting from $\emptyset$ with infinite length, therefore, if we set $\omega$ to be such a simple path, then $\omega|_n\in T_x$ for all $n\in \mathbb Z_{\ge0}$. Therefore, $x=\pi(\omega)\in \bigcap_{n=1}^\infty \Phi_{\tau|_n}(X)$ and $x\in J$.

    Finally, since $T^{(n)}$ is finite, $\bigcup_{\tau\in T^{(n)}}\Phi_\tau(X)$ is compact for all $n$. Therefore, $J$ is compact.
\end{proof}

We now give the definition of conformal TIFSs.
\begin{dfn}\label{dfn:cTIFS}
    A TIFS $\{\phi_\tau\}_{\tau\in T\setminus\{\emptyset\}}$ is a conformal TIFS (for short, CTIFS) if the following conditions are satisfied.
    \begin{enumerate}
    \item (Conformality) $X\subset \mathbb R^d$ for some $d$ and there exists an open set $X\subset V\subset \mathbb R^d$ such that for all $\tau\in T\setminus \{\emptyset\}$, $\phi_\tau$ extends to a $C^1$ diffeomorphism on $V$ and is conformal on $V$.
    \item (Open set condition) For all $\tau\in T$ and $\omega,\rho\in I_\tau$, $\phi_\omega(\interior(X))\cap\phi_\rho(\interior(X))=\emptyset$. Here, $\interior(X)$ denotes the interior of $X$ in $\mathbb R^d$.
    \item (Cone condition) There exist $0<\alpha<\pi$ and $l>0$ such that for all $x\in \partial X$, there exists a unit vector $v\in \mathbb R^d$ such that $\Con(x,\alpha,l,v)\subset \interior(X)$. Here, $\Con(X,\alpha,l,v)$ denotes the open cone with vertex $x$, altitude $l$, angle $\alpha$ and direction $v$, i.e.
    \[
    \Con(x,\alpha,l,v)=\{y\in\mathbb R^d\mid 0<\|y-x\|<l,\,(y-x,v)>\cos\alpha\},
    \]
    where $(\cdot,\cdot)$ denotes the inner product in $\mathbb R^d$ and $\|\cdot\|$ denotes the Euclidean norm.
    \item (Bounded distortion property) There exists $K\ge 1$ such that for all $\tau\in T$ and $x,y\in V$, $|\Phi_\tau'(x)|\le K|\Phi_\tau'(y)|$. Here, $|\Phi_\tau'(x)|$ denotes the operator norm of the derivative of $\Phi_\tau$ at $x$.
    \end{enumerate}
\end{dfn}

\begin{dfn}\label{Def:ma}
    Let $(P,\le)$ be a partially ordered set. We say that $A\subset P$ is an antichain in $P$ if for all $x,y\in A$ with $x\ne y$, neither $x\le y$ nor $y\le x$ holds. We say that $A\subset P$ is a maximal antichain in $P$ if $A$ is an antichain which is maximal with respect to set inclusion, i.e. if $A\subset B\subset P$ and $B$ is an antichain, then $A=B$.
\end{dfn}
\begin{dfn}
Let $x,y\in T$. We say that $x$ and $y$ are comparable if either $x\preceq y$ or $y\preceq x$ is satisfied. We say that $x$ and $y$ are incomparable if these are not comparable.
\end{dfn}

Let $[\rho]|_n=\{\tau\in T^{(n)}\mid \tau\succeq\rho\}$ for $\rho\in T$ and $n\ge |\rho|$.
\begin{lem}\label{Lem:Cover}
    Let $\rho\in T$ and $n>|\rho|$. Then, $A\subset T_{\succ\rho}^{(\le n)}$ is a maximal antichain if and only if $[\rho]|_n$ is the disjoint union of $\{[\omega]|_n\}_{\omega\in A}$.
\end{lem}
\begin{proof}
We first suppose that $A$ is a maximal antichain.

Suppose $\omega,\omega'\in A$ and $[\omega]|_n\cap[\omega']|_n\ne \emptyset$. Then, we can take $\tau\in T^{(n)}$ such that $\omega,\omega'\preceq\tau$, and it follows that $\omega$ and $\omega'$ are comparable. Since $A$ is an antichain, $\omega=\omega'$. 

Suppose $\tau\in [\rho]|_n$. By the maximality of $A$, there exists $\omega\in A$ such that $\tau$ and $\omega$ are comparable. Since $|\omega|\le n=|\tau|$, $\omega\preceq\tau$. Therefore, $\tau\in [\omega]|_n$.

Thus, we have shown that $[\rho]|_n$ is the disjoint union of $\{[\omega]|_n\}_{\omega\in A}$.

We next suppose that $\{[\omega]|_n\}_{\omega\in A}$ is a partition of $[\rho]|_n$.

Let $\omega,\omega'\in A$. Suppose these are comparable, and without loss of generality, suppose $\omega\preceq\omega'$. Since $T$ is pruned, there is $\tau\in T^{(n)}$ such that $\tau\succeq \omega'$. Therefore, $\tau\in[\omega]|_n\cap[\omega']|_n$ and $[\omega]|_n\cap[\omega']|_n\ne\emptyset$. Since $\{[\tau]|_n\}_{\tau\in A}$ are mutually disjoint, $\omega=\omega'$. Therefore, $A$ is an antichain.

To show the maximality of $A$, let $\tau\in T_{\succ\rho}^{(\le n)}$. Then, since $T$ is pruned, there exists $\tau'\in T^{(n)}$ such that $\tau\preceq \tau'$. Since $\tau'\in [\rho]|_n=\bigcup_{\omega\in A}[\omega]|_n$, there is $\omega\in A$ such that $\omega\preceq\tau'$. It follows that $\omega$ and $\tau$ are comparable. Therefore, $A$ is a maximal antichain. Thus, we have proved our lemma.
\end{proof}

\begin{lem}\label{Lem:Rep}
    Let $\rho\in T$, $n>|\rho|$ and $A$ be a subset of $T_{\succ\rho}^{(\le n)}$. Let $\omega\in A$, $|\omega|<n$ and let $B$ be a maximal antichain in $T_{\succeq\omega}^{(\le n)}$. Then, $A$ is a maximal antichain in $T_{\succ\rho}^{(\le n)}$ if and only if $(A\setminus \{\omega\})\cup B$ is a maximal antichain in $T_{\succ\rho}^{(\le n)}$.
\end{lem}
\begin{proof}
Since $B$ is an antichain in $T_{\succeq\omega}^{(\le n)}$, we have that either $B=\{\omega\}$ or $B\subset T_{\succ\omega}^{(\le n)}$. If $B=\{\omega\}$, the claim of our lemma is obvious. Suppose that $B\subset T_{\succ\omega}^{(\le n)}$. Then, by Lemma \ref{Lem:Cover}, $[\omega]|_n$ is the disjoint union of $\{[\tau]|_n\}_{\tau\in B}$. Therefore, $[\rho]|_n$ is the disjoint union of $\{[\tau]|_n\}_{\tau\in A}$ if and only if $[\rho]|_n$ is the disjoint union of $\{[\tau]|_n\}_{\tau\in (A\setminus\{\omega\})\cup B}$. The result follows from Lemma \ref{Lem:Cover}.
\end{proof}

\begin{dfn}\label{dfn:Z*}
    Let $\tifs$ be a CTIFS. For each $\rho\in T$, $n>|\rho|$ and $t>0$, we set
    \begin{equation}\label{Def:Z*}
    Z_{\succ\rho,n}^*(t)=\min\left\{\sum_{\omega\in A}\|\Phi_\omega'\|^t\mid A\text{ is a maximal antichain in }T_{\succ\rho}^{(\le n)}\right\},
    \end{equation}
    \begin{equation}
    Z_n^*(t)=Z_{\succ\emptyset,n}^*(t).
    \end{equation}

    Here, we set
    \begin{equation}
    \|\Phi_\omega'\|=\sup_{x\in X}|\Phi_\omega'(x)|.
    \end{equation}
\end{dfn}

\begin{prop}
    Let $t\ge0$. Then, $Z_n^*(t)$ are monotonically decreasing with respect to $n$.
\end{prop}
\begin{proof}
    It is sufficient to show that if $A$ is a maximal antichain in $T_{\succ\rho}^{(\le n)}$, then $A$ is also a maximal antichain in $T_{\succ\rho}^{(\le n+1)}$.

    Let $A$ be a maximal antichain in $T_{\succ\rho}^{(\le n)}$, and let $\omega$ be an arbitrary element of $T_{\succ\rho}^{(\le n+1)}$. It is sufficient to show that $\omega$ is comparable with some elements of $A$. If $\omega$ is in $T_{\succ\rho}^{(\le n)}$, this is obvious from the maximality of $A$ in $T_{\succ\rho}^{(\le n)}$. Suppose $\omega\in T_{\succ\rho}^{(n+1)}$. In this case, $\omega|_n$ is comparable with some elements $\tau\in A$. Since $|\tau|\le n$, we can see $\tau\preceq \omega|_n\preceq\omega$. Therefore, $\tau$ and $\omega$ are comparable.
\end{proof}
\begin{dfn}
    Let $\Psi=\tifs$ be a CTIFS. We define
    \begin{equation}
    Z_\infty^*(t)=\lim_{n\to\infty}Z_n^*(t),
    \end{equation}
    \begin{equation}
    \beta^*(\Psi)=\inf\{t\ge0\mid Z_\infty^*(t)=0\}.
    \end{equation}
\end{dfn}

The following is the main theorem of this paper.
\begin{thm}\label{Thm:Main}
Let $\Psi=\tifs$ be a CTIFS and $J$ be the limit set. Let $t\ge 0$. Then, we have the following.
\begin{itemize}
\item[{\em (i)}] If $Z_\infty^*(t)=0$, then $\dim_H(J)\le t$.
\item[{\em (ii)}] In addition to the assumptions of our theorem, suppose
\begin{equation}\label{Eq:Cond1}
\lim_{|\tau|\to\infty}\frac1{|\tau|}\log\#I_\tau=0.
\end{equation}
Then, if $Z_\infty^*(t)>0$, we have that $\dim_H(J)\ge t$.
\end{itemize}
Here, $\dim_H(J)$ denotes the Hausdorff dimension of $J$. Especially, if the condition \eqref{Eq:Cond1} in {\em (ii)} is satisfied, then $\dim_H(J)=\beta^*(\Psi)$.
\end{thm}

The proof of Theorem \ref{Thm:Main} is given in section \ref{Sec:Proof}.

\section{Constructing a probability measure}\label{Sec:Measure}
In this section, we prove the existence of nice measures $\mu_{t,n}^*$ on $T^{(n)}$ (Theorem \ref{Thm:Gib}), which we need for proving Theorem \ref{Thm:Main}.

In the following, we fix $t\ge0$ and $n\in\mathbb N$.

\begin{ntn}
For each $\rho\in T$ with $|\rho|<n$, let $A_{\rho}$ be a maximal antichain in $T_{\succ\rho}^{(\le n)}$ which attains the minimum of \eqref{Def:Z*}. In the following, for each $\rho\in T$ with $|\rho|<n$, we fix such $A_\rho$. 
\end{ntn}

\begin{dfn}
We say that a sequence $\rho_0,\rho_1,\ldots,\rho_k$ of $T^{(\le n)}$ is a chain if $\rho_0=\emptyset$ and $\rho_{i+1}\in A_{\rho_i}$ for all $0\le i<k$. We say that a chain $\rho_0,\rho_1,\ldots,\rho_k$ of $T^{(\le n)}$ is a $\tau$-chain if $\rho_k=\tau$. We set
\begin{equation}
    A^*=\{\tau\in T^{(\le n)}\mid \text{there exists a $\tau$-chain $\rho_0,\ldots,\rho_k$}\}.
\end{equation}
\end{dfn}

\begin{lem}\label{Lem:Chain}
    In the above settings, we have the following.
    \begin{enumerate}
    \item For any chain $\rho_0,\ldots,\rho_k$, we have $\rho_0\prec \rho_1\prec\cdots\prec\rho_k$.
    \item For each $\tau\in A^*$, there exists a unique $\tau$-chain.
    \item $T^{(n)}\subset A^*$.
    \end{enumerate}
\end{lem}
\begin{proof}
    1. is obvious.

    To show 2., let $\rho_0,\ldots,\rho_k$ and $\rho'_0,\ldots,\rho'_{k'}$ be $\tau$-chains, and suppose $\rho_i=\rho'_i\prec \tau$. Then, $\rho_{i+1},\rho'_{i+1}\in A_{\rho_i}$. Since $\rho_{i+1},\rho'_{i+1}\preceq\tau$, these are both comparable. Therefore, since $A_{\rho_i}$ is an antichain, $\rho_{i+1}=\rho'_{i+1}$. By mathematical induction, the proof is done.

    To show 3., let $\tau\in T^{(n)}$. We construct a $\tau$-chain as follows. First, let $\rho_0=\emptyset$. Suppose $\rho_0,\ldots,\rho_i$ are defined as a chain and satisfy $\rho_i\preceq\tau$. If $\rho_i$ is $\tau$, we finish the construction. Otherwise, since $\tau\succ\rho_i$ and $A_{\rho_i}$ is a maximal antichain in $T_{\succ\rho}^{(\le n)}$, $\tau$ is comparable with some $\rho \in A_{\rho_i}$. Since $|\tau|=n$, $\rho\preceq\tau$. We set $\rho_{i+1}=\rho$.

    Repeating the above procedure at most $n$ times, we obtain a $\tau$-chain $\rho_0,\ldots,\rho_k$, and the lemma is proved.
\end{proof}

\begin{dfn}
For each $\tau\in A^*$, we take the $\tau$-chain $\rho_0,\ldots,\rho_k$ and let
\begin{equation}
m^*(\tau)=\prod_{i=0}^{k-1}\frac{\|\Phi_{\rho_{i+1}}'\|^t}{Z_{\succ\rho_i,n}^*(t)}.
\end{equation}
We also set $m^*(\emptyset)=1$.
\end{dfn}

\begin{lem}\label{Lem:Dist}
    For each $\rho\in A^*$,
    \begin{equation}\label{Eq:Dist}
    \sum_{\tau\in[\rho]|_n}m^*(\tau)=m^*(\rho).
    \end{equation}
\end{lem}
\begin{proof}
    We show the lemma by the backward induction with respect to $|\rho|$. If $|\rho|=n$, \eqref{Eq:Dist} is obvious.

    Let $\rho\in A^*$ and suppose \eqref{Eq:Dist} is correct for all $\rho'\in A^*$ with $|\rho'|>|\rho|$. By Lemma \ref{Lem:Cover} and the induction hypothesis,
    \begin{equation}\label{Eq:Part}
    \sum_{\tau\in[\rho]|_n}m^*(\tau)=\sum_{\omega\in A_{\rho}}\sum_{\tau\in[\omega]|_n}m^*(\tau)=\sum_{\omega\in A_\rho}m^*(\omega).
    \end{equation}

    Let $\rho_0,\ldots,\rho_k$ be a $\rho$-chain. Then, for each $\omega\in A_\rho$, the sequence $\rho_0,\ldots,\rho_k,\omega$ is an $\omega$-chain. Therefore, by definition of $A_\rho$,
    \begin{align}
    \sum_{\omega\in A_\rho}m^*(\omega)&=\sum_{\omega\in A_\rho}\left(\prod_{i=0}^{k-1}\frac{\|\Phi_{\rho_{i+1}}'\|^t}{Z_{\succ\rho_i,n}^*(t)}\right)\cdot\frac{\|\Phi_{\omega}'\|^t}{Z_{\succ\rho,n}^*(t)}\\
    &=\left(\prod_{i=0}^{k-1}\frac{\|\Phi_{\rho_{i+1}}'\|^t}{Z_{\succ\rho_i,n}^*(t)}\right)\cdot\frac{1}{Z_{\succ\rho,n}^*(t)}\sum_{\omega\in A_\rho}\|\Phi_\omega'\|^t\\
    &=m^*(\rho).
    \end{align}
    Together with \eqref{Eq:Part}, we obtain \eqref{Eq:Dist}. By mathematical induction, the proof is done.
\end{proof}

We now prove the existence of nice measures $\mu_{t,n}^*$ on $T^{(n)}$.
\begin{thm}\label{Thm:Gib}
    Let $\tifs$ be a CTIFS. For all $t$ and $n$, there exists a probability measure $\mu_{t,n}^*$ on $T^{(n)}$ such that for all $\omega\in T^{(\le n)}\setminus \{\emptyset\}$,
    \begin{equation}\label{Eq:Ineq}
    \mu_{t,n}^*([\omega]|_n)\le \frac{\|\Phi_\omega'\|^t}{Z_n^*(t)},
    \end{equation}
    where $[\omega]|_n=\{\tau\in T^{(n)}\mid \tau\succeq\omega\}$.
\end{thm}

\begin{proof}
    Let $\mu^*=\mu_{t,n}^*$ be a measure on $T^{(n)}$ defined by
    \begin{equation}
    \mu^*(\{\tau\})=m^*(\tau),\quad \tau\in T^{(n)}.
    \end{equation}
    
    By Lemma \ref{Lem:Dist}, for any $\omega\in A^*$,
    \begin{equation}
    \mu^*([\omega]|_n)=m^*(\omega).
    \end{equation}
    Especially, $\mu^*(T^{(n)})=m^*(\emptyset)=1$, therefore $\mu^*$ is a probability measure.

    To show \eqref{Eq:Ineq}, let $\omega\in T^{(\le n)}\setminus\{\emptyset\}$. Let $\rho\in A^*$ be the element with maximum height such that $\rho\prec\omega$. Let $A_\omega^*=\{\tau\in A_\rho\mid \tau\ge\omega\}$.

    We first show that $A_\omega^*$ is a maximal antichain in $T_{\succeq \omega}^{(\le n)}$. Since $A_\omega^*\subset A_\rho$, $A_\omega^*$ is an antichain. To see the maximality, let $\tau \in T_{\succeq\omega}^{(\le n)}$. Since $A_\rho$ is a maximal antichain in $T_{\succ\rho}^{(\le n)}$ and $\tau\in T_{\succ\rho}^{(\le n)}$, there is $\tau'\in A_\rho$ which is comparable with $\tau$. Then $\tau'$ is also comparable with $\omega$. From the maximality of the height of $\rho$, it follows that $\tau'\not\prec \omega$. Therefore, $\tau'\in A_\omega^*$. Thus, $A_\omega^*$ is a maximal antichain in $T_{\succeq \omega}^{(\le n)}$.

    By lemma \ref{Lem:Cover},
    \begin{equation}\label{Eq:Eval1}
    \mu^*([\omega]|_n)=\sum_{\tau\in[\omega]|_n}m^*(\tau)=\sum_{\tau\in A_{\omega}^*}\sum_{\tau'\in[\tau]|_n}m^*(\tau')=\sum_{\tau\in A_\omega^*}m^*(\tau).
    \end{equation}

    From Lemma \ref{Lem:Rep} (with $A=(A_\rho\setminus A_\omega^*)\cup\{\omega\}$ and $B=A_\omega^*$), $(A_\rho\setminus A_\omega^*)\cup\{\omega\}$ is a maximal antichain in $T_{\succ\rho}^{(\le n)}$. Therefore, by definition of $Z_{\succ\rho,n}^*(t)$, we have
    \begin{equation}
    \sum_{\tau\in A_\rho}\|\Phi_\tau'\|^t\le \sum_{\tau\in (A_\rho\setminus A_\omega^*)\cup\{\omega\}}\|\Phi_\tau'\|^t,
    \end{equation}
    which implies that 
    \begin{equation}
    \sum_{\tau\in A_\omega^*}\|\Phi_\tau'\|^t\le \|\Phi_\omega'\|^t.
    \end{equation}

    Let $\rho_0,\ldots,\rho_k$ be a $\rho$-chain. Then, for each $\tau\in A_\omega^*$, the sequence $\rho_0,\ldots,\rho_k,\tau$ is a $\tau$-chain. Therefore,
    \begin{align}
    \sum_{\tau\in A_\omega^*}m^*(\tau)&=\sum_{\tau\in A_\omega^*}\left(\prod_{i=1}^{k-1}\frac{\|\Phi_{\rho_{i+1}}'\|^t}{Z_{\succ\rho_i,n}^*(t)}\right)\cdot \frac{\|\Phi_{\tau}'\|^t}{Z_{\succ\rho,n}^*(t)}\\
    &\le \left(\prod_{i=0}^{k-1}\frac{\|\Phi_{\rho_{i+1}}'\|^t}{Z_{\succ\rho_i,n}^*(t)}\right)\cdot \frac{\|\Phi_\omega'\|^t}{Z_{\succ\rho,n}^*(t)}.\label{Eq:Eval2}
    \end{align}

    Let $1\le i\le k$. Since $A_{\rho_i}$ is a maximal antichain in $T_{\succ\rho_i}^{(\le n)}$, we have that $A_{\rho_i}$ is a maximal antichain in $T_{\succeq\rho}^{(\le n)}$. Thus, Lemma \ref{Lem:Rep} implies that $(A_{\rho_{i-1}}\setminus \{\rho_{i}\})\cup A_{\rho_{i}}$ is a maximal antichain in $T_{\succ\rho_{i-1}}^{(\le n)}$. Therefore, we have
    \begin{equation}
    \sum_{\tau\in A_{\rho_{i-1}}}\|\Phi_\tau'\|^t\le \sum_{\tau\in(A_{\rho_{i-1}}\setminus\{\rho_i\})\cup A_{\rho_i}}\|\Phi_\tau'\|^t,
    \end{equation}
    which implies that
    \begin{equation}\label{Eq:Eval3}
    \|\Phi_{\rho_i}'\|^t\le \sum_{\tau\in A_{\rho_i}}\|\Phi_\tau'\|^t=Z_{\succ\rho_i,n}^*(t).
    \end{equation}
    Combining \eqref{Eq:Eval1}, \eqref{Eq:Eval2} and \eqref{Eq:Eval3},
    \begin{align}
    \mu^*([\omega]|_n)&\le \left(\prod_{i=0}^{k-1}\frac{\|\Phi_{\rho_{i+1}}'\|^t}{Z_{\succ\rho_i,n}^*(t)}\right)\cdot \frac{\|\Phi_\omega'\|^t}{Z_{\succ\rho,n}^*(t)}\\
    &=\left(\prod_{i=1}^{k}\frac{\|\Phi_{\rho_{i}}'\|^t}{Z_{\succ\rho_i,n}^*(t)}\right)\cdot \frac{\|\Phi_\omega'\|^t}{Z_{\succ\emptyset,n}^*(t)}\\
    &\le \frac{\|\Phi_\omega'\|^t}{Z_{\succ\emptyset,n}^*(t)}.
    \end{align}
    Thus, we have proved our theorem.
\end{proof}

\section{Proof of theorem \ref{Thm:Main}}\label{Sec:Proof}
In this section, we prove Theorem \ref{Thm:Main}.

For simplicity, we use the following notations. For each $\tau\in T$ and $n\in \mathbb Z_{\ge0}$, we set
\begin{equation}
X_\tau =\Phi_\tau(X),
\end{equation}
\begin{equation}
X_n=\bigcup_{\tau\in T^{(n)}}X_\tau.
\end{equation}

We first give the following lemma, which follows from the Bounded Distortion Property. 

\begin{lem}\label{Lem:BDP}
    Let $\tifs$ be a CTIFS. Then, there exists $K\ge 1$ such that, for all $\tau \in T$,
    \begin{equation}\label{Eq:BDP1}
    \diam(X_\tau)\le K\|\Phi_\tau'\|,
    \end{equation}
    \begin{equation}\label{Eq:BDP2}
    \|\Phi_\tau'\|\le K\diam(X_\tau).
    \end{equation}
\end{lem}
\begin{proof}
Let $K'$ be a constant of the Bounded Distortion Property. Let $\{U_i\}_{i=1}^n$ be an open cover of $X$ such that each $U_i$ is convex, $U_i\subset V$ and $\diam(U_i)\le1$. Take $x,y\in X$ arbitrarily. By the connectivity of X, we can take $i_1,\ldots,i_{k+1}$ with $k+1\le n$ and $x_1,\ldots, x_{k}$ such that $x\in U_{i_1}$, $y\in U_{i_k}$, $x_{j}\in U_{i_j}\cap U_{i_{j+1}}$ for each $1\le j\le k$, and $i_j\ne i_{j'}$ for $j\ne j'$. Therefore, if we write $x=x_0$, $y=x_{k+1}$, then for each $\omega\in T$,
\begin{align}
\|\Phi_\omega(y)-\Phi_\omega(x)\|&\le \sum_{j=1}^k\|\Phi_\omega(x_{j+1})-\Phi_\omega(x_{j})\|\\
&\le \sum_{j=1}^kK'\|\Phi_\omega'\|\|x_{j+1}-x_j\|\\
&\le kK'\|\Phi_\omega'\|\\
&\le nK'\|\Phi_\omega'\|.
\end{align}

Letting $K\ge nK'$, we get \eqref{Eq:BDP1}.

To show \eqref{Eq:BDP2}, take $x_0\in X$ and $r>0$ such that $\overline{B(x_0,r)}\subset \interior(X)$. Here, $B(x_0,r)$ is the open ball which is centered at $x_0$ and has a radius $r$. Let $\omega\in T$. Let $R$ be the maximal number such that $B(\Phi_\omega(x_0),R)\subset \Phi_\omega(B(x_0,r))$. Then, for any small $\varepsilon>0$, there exists $x\in B(\Phi_\omega(x_0),R+\varepsilon)$ such that
\begin{equation}
\Phi_\omega^{-1}(x)\in X\setminus B(x_0,r).
\end{equation}
Thus, $\|\Phi_\omega^{-1}(x)-x_0\|>r$. On the other hand,
\begin{equation}
\|\Phi_\omega^{-1}(x)-x_0\|\le \sup_{y\in B(\Phi_\omega(x_0),R+\varepsilon)}|(\Phi_\omega^{-1})'(y)|(R+\varepsilon)\le K'\|\Phi_\omega'\|^{-1}(R+\varepsilon).
\end{equation}
Therefore, $r\le K'\|\Phi_\omega'\|^{-1}(R+\varepsilon)$. Since $\varepsilon>0$ is arbitrary,
\begin{equation}
R\ge (K')^{-1}\|\Phi_\omega'\|r.
\end{equation}
It follows that
\begin{equation}
B(\Phi_{\omega}(x_0),(K')^{-1}\|\Phi_\omega'\|r)\subset \Phi_\omega(B(x_0,r)).
\end{equation}
Hence,
\begin{equation}
\diam(\Phi_\omega(X))\ge \diam(\Phi_\omega(B(x_0,r)))\ge (K')^{-1}\|\Phi_\omega'\|r.
\end{equation}
Letting $K\ge K'r^{-1}$, we get \eqref{Eq:BDP2}.
\end{proof}

We use the following lemma which we can prove similarly as \cite[Lemma 4.2.6]{MR2003772}.

\begin{lem}\label{Lem:Fin}
Let $\tifs$ be a CTIFS. Then, there exists $M\ge 1$ such that the following is satisfied.
\begin{itemize}
\item Let $x\in X$, $r>0$. Let $W\subset T$ be an antichain. Suppose that for each $\tau \in W$, $\diam(X_\tau)\ge r$ and $X_\tau\cap B(x,r)\ne \emptyset$. Then $\# W\le M$.
\end{itemize}

\end{lem}

Finally, we need the following well-known fact. For the proof, see \cite{MR1102677}.

\begin{thm}\label{Thm:Frostman}
Let $J\subset \mathbb R^d$ and let $\mu$ be a probability measure on $J$. If there is a constant $C$ such that for all $x\in J$,
\begin{equation}
\limsup_{r\to0}\frac{\mu(B(x,r))}{r^t}\le C,
\end{equation}
then $\dim_H(J)\ge t$.
\end{thm}



We need to construct the following nice probability measures on $X$ to prove Theorem \ref{Thm:Main} (ii).
\begin{dfn}
For each $t\ge0$ and $n\in \mathbb N$, we set,
\begin{equation}\label{Eq:Def_mu}
\mu_{t,n}(A):=\sum_{\tau\in T^{(n)}}\mu_{t,n}^*(\{\tau\})\frac{\lambda_d(A\cap X_\tau)}{\lambda_d(X_\tau)},
\end{equation}
where $\mu_{t,n}^*$ is the probability measure on $T^{(n)}$ in Theorem \ref{Thm:Gib}. Obviously, $\mu_{t,n}$ is a probability measure.
\end{dfn}

\begin{lem}\label{Lem:Gib_X}
Let $\tifs$ be a CTIFS and suppose $Z_\infty^*(t)>0$. Let $\mu_{t,n}$ be the probability measures defined by \eqref{Eq:Def_mu}. Then, for each $n\in\mathbb Z_{\ge 0}$ and $\omega\in T^{(\le n)}$,
\begin{equation} \label{Eq:Gib}
\mu_{t,n}(X_\omega)\le \frac{\|\Phi_\omega'\|^t}{Z_\infty^*(t)}.
\end{equation}
\end{lem}
\begin{proof}
To show \eqref{Eq:Gib}, let $\omega\in T^{(\le n)}$. Let $\tau\in T^{(n)}$. If $\omega\not\preceq\tau$, then there exists $k$ such that $\omega|_k=\tau|_k$ and $\omega|_{k+1}\ne\tau|_{k+1}$. By the Open Set Condition, $\phi_{\omega|_{k+1}}(\interior(X))\cap\phi_{\tau|_{k+1}}(\interior(X))=\emptyset$. Since $\Phi_{\omega|_k}=\Phi_{\tau|_k}$ is injective,
\begin{equation}
\Phi_{\omega|_{k+1}}(\interior(X))\cap\Phi_{\tau|_{k+1}}(\interior(X))=\emptyset.
\end{equation}
Therefore,
\begin{equation}
\Phi_{\omega}(\interior(X))\cap\Phi_{\tau}(\interior(X))\subset \Phi_{\omega|_{k+1}}(\interior(X))\cap\Phi_{\tau|_{k+1}}(\interior(X))=\emptyset.
\end{equation}
Furthermore, since $\Phi_\omega$ and $\Phi_\tau$ are homeomorphisms,
\begin{equation}
\interior(X_\omega)\cap\interior(X_\tau)=\emptyset.
\end{equation}
Therefore,
\begin{equation}
X_\omega\cap X_\tau\subset \partial X_\omega\cup \partial X_\tau.
\end{equation}

Suppose that $\lambda_d(\partial X)>0$. Then, by the Lebesgue's density theorem,
\begin{equation}
\lim_{r\to 0}\frac{\lambda_d(B(x,r)\cap \partial X)}{\lambda_d(B(x,r))}=1\quad {\rm a.e.}\quad x\in\partial X.
\end{equation}
Moreover, by the Cone Condition,
\begin{equation}
\limsup_{r\to0}\frac{\lambda_d(B(x,r)\cap \partial X)}{\lambda_d(B(x,r))}\le 1-\lambda_d(\Con(0,\alpha, 1, u))
\end{equation}
for all $x\in \partial X$. Here, $u$ is an arbitrary unit vector and $\alpha$ is the angle of the Cone Condition. However, this is a contradiction. Therefore, $\lambda_d(\partial X)=0$.

Since $\Phi_\omega,\Phi_\tau$ are diffeomorphisms,
\begin{equation}
\lambda_d(\partial X_\tau)=\lambda_d(\partial X_\omega)=0,
\end{equation}
\begin{equation}
\lambda_d(X_\omega\cap X_\tau)=0.
\end{equation}

If $\omega\preceq\tau$, then $\lambda_d(X_\omega\cap X_\tau)=\lambda_d(X_\tau)$. Therefore, for each $t\ge0$, 
\begin{equation}
\mu_{t,n}(X_\omega)=\sum_{\tau\in T^{(n)}}\mu_{t,n}^*(\{\tau\})\frac{\lambda_d(X_\omega\cap X_\tau)}{\lambda_d(X_\tau)}=\mu_{t,n}^*([\omega]|_n)\le \frac{\|\Phi_\omega'\|^t}{Z_n^*(t)}\le \frac{\|\Phi_\omega'\|^t}{Z_\infty^*(t)}.
\end{equation}

Thus, we have proved our lemma.
\end{proof}

The following lemma gives an upper bound of the density of $\mu_{t,n}$, which enables us to apply Theorem \ref{Thm:Frostman}.

\begin{lem}\label{Lem:Dens}
Let $\tifs$ be a CTIFS, and let $0\le t'<t$. For each $n\in\mathbb N$, let $\mu_{t,n}$ be the probability measure defined by \eqref{Eq:Def_mu}. Suppose that $Z_\infty^*(t)>0$. Suppose also that the condition \eqref{Eq:Cond1} in Theorem \ref{Thm:Main} is satisfied. Then, there exist $C>0$ and $r_0>0$ such that for any $x\in X$ and $0<r\le r_0$, there exists $n$ such that for any $q\ge n$,
\begin{equation}\label{Eq:Dens}
\mu_{t,q}(B(x,r))\le Cr^{t'}.
\end{equation}
\end{lem}
\begin{proof}
Let $s\in(0,1)$ be the number for $\tifs$ coming from Definition \ref{dfn:TIFS}. Let $n_0\ge 1$ be large enough so that for any $\tau\in T$ with $|\tau|\ge n_0$,
\begin{equation}
\frac1{|\tau|}\log\#I_\tau<-\frac12(t-t')\log s,
\end{equation}
which is equivalent to
\begin{equation}\label{Eq:Exp1}
\#I_\tau<s^{-\frac12|\tau|(t-t')}.
\end{equation}

We set
\begin{equation}
r_0:=\min\{\diam(X_\tau)\mid |\tau|\le n_0\}.
\end{equation}

Take $0<r\le r_0$ and $x\in X$ arbitrarily. Let $W$ be the set of all elements $\tau$ of $T$ which satisfy the following.
\begin{itemize}
\item $\diam(X_\tau)\ge r$.
\item There exists $a\in I_\tau$ such that $X_a\cap B(x,r)\ne\emptyset$ and $\diam(X_a)<r$.
\end{itemize}
Let $W'$ be the set of all maximal elements of $W$ with respect to $\preceq$. For each $\tau\in W'$, let $n(\tau)=\min_{\omega\preceq\tau,\omega\in W}|\omega|$.

Let $n$ be a positive integer such that $s^n\diam(X)<r$. Since $\diam(X_\tau)< s^{|\tau|}\diam(X)$ for any $\tau\in T$, we have that $\diam(X_\tau)<r$ if $|\tau|\ge n$. Therefore, if $\tau\in W$, then $|\tau|<n$. Take $q\ge n$ arbitrarily.

For each $\tau\in T$, let $B_\tau$ be the set defined by
\begin{equation}
B_\tau=\{a\in I_\tau\mid X_\alpha\cap B(x,r)\ne \emptyset,\,\diam(X_a)<r\}.
\end{equation}

Suppose $y\in X_q\cap B(x,r)$. Then there is $\tau\in T^{(q)}$ such that $y\in X_\tau$. Since $\diam(X_\tau)<r$, there is $k$ such that $\diam(X_{\tau|_k})\ge r$ and $\diam(X_{\tau|_{k+1}})<r$. Especially, $\tau|_k\in W$ and $\tau|_{k+1}\in B_{\tau|_k}$. Since the height of each element of $W$ is uniformly bounded, there exists some $\tau'\in W'$ such that $\tau|_k\preceq \tau'$. By definition of $n(\tau')$, we have that $k\ge n(\tau')$. Therefore,
\begin{equation}
y\in \bigcup_{\tau'\in W'}\bigcup_{i=n(\tau')}^{|\tau'|}\bigcup_{a\in B_{\tau'|_{i}}}X_a.
\end{equation}
Therefore,
\begin{equation}
X_q\cap B(x,r)\subset \bigcup_{\tau\in W'}\bigcup_{i=n(\tau)}^{|\tau|}\bigcup_{a\in B_{\tau|_{i}}}X_a.
\end{equation}
Since $\mu_{t,q}$ has a support on $X_q$,
\begin{equation}\label{Eq:Eval_mu1}
\mu_{t,q}(B(x,r))\le \sum_{\tau\in W'}\sum_{i=n(\tau)}^{|\tau|}\sum_{a\in B_{\tau|_i}}\mu_{t,q}(X_a).
\end{equation}

If $\omega\in \bigcup_{\tau\in W'}\{\tau|_{n(\tau)},\ldots,\tau|_{|\tau|}\}$ and $a\in B_\omega$, then $|\omega|<q$ and $|a|\le q$. Therefore, by Lemma \ref{Lem:Gib_X} and Lemma \ref{Lem:BDP},
\begin{align}
\sum_{a\in B_{\omega}}\mu_{t,q}(X_a)
&\le \sum_{a\in B_{\omega}}\frac{\|\Phi_a'\|^t}{Z_\infty^*(t)}\\
&\le \frac1{Z_\infty^*(t)}\sum_{a\in B_{\omega}}\|\Phi_a'\|^{t-t'}\|\Phi_a'\|^{t'}\\
&\le \frac1{Z_\infty^*(t)}\sum_{a\in B_{\omega}}s^{|a|(t-t')}(K\diam(X_a))^{t'}\\
&\le \frac{K^{t'}}{Z_\infty^*(t)}\sum_{a\in B_{\omega}}s^{(|\omega|+1)(t-t')}r^{t'}\\
&\le \frac{K^{t'}}{Z_\infty^*(t)}\#B_{\omega}s^{|\omega|(t-t')}r^{t'}.
\end{align}
Moreover, by \eqref{Eq:Exp1},
\begin{equation}
\#B_\omega\le \#I_\omega\le s^{-\frac12|\omega|(t-t')}.
\end{equation}
Therefore,
\begin{equation}\label{Eq:Eval_mu2}
\sum_{a\in B_{\omega}}\mu_{t,q}(X_a)
\le \frac{K^{t'}}{Z_\infty^*(t)}s^{\frac12|\omega|(t-t')}r^{t'}.
\end{equation}

Since $s^{\frac12(t-t')}<1$, we have
\begin{equation}\label{Eq:Eval_mu3}
\sum_{i=n(\tau)}^{|\tau|}s^{\frac12 i(t-t')}<\sum_{i=1}^\infty s^{\frac12 i(t-t')}=\frac1{1-s^{\frac12(t-t')}}.
\end{equation}

Since $W'$ is an antichain, Lemma \ref{Lem:Fin} implies that
\begin{equation}\label{Eq:Eval_mu4}
\#W'\le M.
\end{equation}
Combining \eqref{Eq:Eval_mu1}, \eqref{Eq:Eval_mu2}, \eqref{Eq:Eval_mu3} and \eqref{Eq:Eval_mu4},
\begin{equation}
\mu_{t,q}(B(x,r))\le \frac{K^{t'}M}{Z_\infty^*(t)(1-s^{\frac12(t-t')})}r^{t'}.
\end{equation}
If we set 
\begin{equation}
C=\frac{K^{t'}M}{Z_\infty^*(t)(1-s^{\frac12(t-t')})},
\end{equation}
then \eqref{Eq:Dens} is satisfied. Thus, we have proved our lemma.
\end{proof}

We now prove Theorem \ref{Thm:Main}.

\begin{proof}[proof of Theorem \ref{Thm:Main}]
    We first prove Theorem \ref{Thm:Main} (i).
    
    Suppose $Z_\infty^*(t)=0$. Take $\varepsilon>0$ arbitrarily. Then, if $n$ is large enough, $Z_n^*(t)<\varepsilon$. Therefore, there is a maximal antichain $A\subset T^{(n)}$ such that,
    \begin{equation}
    \sum_{\tau\in A}\|\Phi_\tau'\|^t<\varepsilon.
    \end{equation}

    Now, Let $x\in J$. Then, there is some $\omega\in \partial T$ such that $x=\pi(\omega)$. Since $A$ is a maximal antichain, there exists $\tau\in A$ such that $\tau\preceq\omega|_n$. Therefore,
    \begin{equation}
    x=\pi(\omega)\in X_{\omega|n}\subset X_\tau.
    \end{equation}
    It is shown that
    \begin{equation}
    J\subset \bigcup_{\tau\in A}X_\tau.
    \end{equation}
    
    By Lemma \ref{Lem:BDP}, $\diam(X_\tau)\le K\|\Phi_\tau'\|$. Therefore,
    \begin{equation}
    H_{K\varepsilon}^t(J)\le \sum_{\tau\in A}(\diam(X_\tau))^t\le \sum_{\tau\in A}K^t\|\Phi_\tau'\|^t\le K^t\varepsilon.
    \end{equation}
    Here, $H_\delta^t$ denotes the $\delta$ approximation of the $t$-dimensional Hausdorff measure. Since we took $\varepsilon$ arbitrarily, the $t$-dimensional Hausdorff measure of $J$ satisfies
    \begin{equation}
        H^t(J)=0.
    \end{equation}
   Therefore, $\dim_H(J)\le t$.

    We now prove Theorem \ref{Thm:Main} (ii).

    Let $t\ge0$, $n\in\mathbb N$. Let $\mu_{t,n}$ be the probability measure defined by \eqref{Eq:Def_mu}. Since the space of all probability measures on $X$ equipped with weak* topology is sequential compact, We can choose a subsequence $\{\mu_{q_k}\}_{k}$ of $\{\mu_{t,n}\}_n$ and a probability measure $\mu$ on $X$ such that $\mu_{q_k}\to \mu$ as $k\to\infty$ in the weak* topology.

    By the portmanteau lemma, for each $n\in\mathbb Z_{\ge 0}$,
    \begin{equation}
    \mu(X_n)\ge \limsup_{k\to\infty}\mu_{q_k}(X_n)=1,
    \end{equation}
    and therefore,
    \begin{equation}
    \mu(J)=\lim_{n\to\infty}\mu(X_n)=1.
    \end{equation}

    Let $0\leq t'<t$. Let $C>0$ and $r_0>0$ be the numbers coming from Lemma \ref{Lem:Dens}.
    Again, by the portmanteau lemma and Lemma \ref{Lem:Dens}, for any $x\in J$ and $0<r\le r_0$,
    \begin{equation}
    \mu(B(x,r))\le \liminf_{k\to\infty}\mu_{q_k}(B(x,r))\le Cr^{t'}.
    \end{equation}

    Therefore, by Theorem \ref{Thm:Frostman}, $\dim_H(J)\ge t'$. Hence, $\dim_H(J)\geq t$.

    Thus, we have proved Theorem \ref{Thm:Main}.
\end{proof}

\section{Example}\label{Sec:Example}

A non-autonomous IFS (with finite alphabet) $\Psi=\tifs$ is a TIFS which satisfies the following.
\begin{itemize}
\item There exists a sequence of finite sets of alphabets $I^{(1)},I^{(2)},\ldots$ such that
\begin{equation}
T=\bigcup_{n=0}^\infty I^n
\end{equation}
where
\begin{equation}
I^n=I^{(1)}\times\cdots\times I^{(n)}.
\end{equation}
Here, we consider $I^0$ as a one-point set with an only element $\emptyset$.
\item For each $n\in \mathbb N$ and $i\in I^{(n)}$, there is a map $\psi_i$ such that for any $\tau=(\omega_1,\ldots,\omega_n)\in I^n$,
\begin{equation}
\phi_\tau=\psi_{\omega_n}
\end{equation}
\end{itemize}

For a non-autonomous conformal IFS $\Psi=\tifs$, we considered the functions
\begin{equation}
Z_n(t)=\sum_{(\omega_1,\ldots,\omega_n)\in I^{(1)}\times\cdots \times I^{(n)}}\|(\phi_{\omega_1}\circ\cdots\circ\phi_{\omega_n})'\|^t
\end{equation}
and the Hausdorff dimensions of the limit sets were obtained by,
\begin{equation}
\dim_H(J)=\inf\{t\ge0\mid \liminf_{n\to\infty}Z_n(t)=0\}
\end{equation}
(see \cite{MR3449231}).
Therefore, one might expect that for any CTIFS $\Psi$, the Hausdorff dimensions of the limit sets of $\Psi$ can be obtained by the following formula.
\begin{equation}\label{Eq:Cand}
\dim_H(J)=\beta(\Psi),
\end{equation}
where
\begin{equation}
\beta(\Psi)=\inf\{t\ge0\mid \liminf_{n\to\infty}Z_n(t)=0\},
\end{equation}
\begin{equation}\label{Eq:Def_Zn}
Z_n(t)=\sum_{\omega\in T^{(n)}}\|\Phi_\omega'\|^t.
\end{equation}

Unfortunately, this is not the case. In general, for any $\Psi$, we have $\beta^*(\Psi)\le \beta(\Psi)$, since $T^{(n)}$ is a maximal antichain in $T^{(n)}\setminus\{\emptyset\}$ for each $n$. However, it may happen that $\beta^*(\Psi)<\beta(\Psi)$. In this section, we give an example of CTIFS $\Psi=\tifs$ such that the exact Hausdorff dimension of the limit set of $\Psi$ (which is equal to $\beta^{\ast}(\Psi)$) is smaller than $\beta(\Psi)$.

Let $X=[0,1]$, and let
\begin{equation}
\begin{cases}
\psi_{0,\frac12}(x)=\frac12 x,&\psi_{1,\frac12}(x)=\frac12x+\frac12\\
\psi_{0,\frac13}(x)=\frac13x,&\psi_{1,\frac13}(x)=\frac13x+\frac23.
\end{cases}
\end{equation}
Let $T=\{0,1\}^{<\mathbb N}$. The order in $T$ is defined by
\begin{equation}
\omega\preceq\tau\quad\Leftrightarrow\quad \text{$\omega$ is an initial segment of $\tau$}.
\end{equation}
For each $k$, we denote by $1^k$ the element of $\{0,1\}^k$ such that all components are $1$. For each $\omega\in \{0,1\}^{<\mathbb N}$, let $i(\omega)=\omega_{|\omega|}$.

Let $\{t_k\}_k$ be a strictly increasing sequence in $(0,1)$ such that,
\begin{equation}
\lim_{k\to\infty}t_k=1.
\end{equation}
Let $\{n_k\}_k$ also be a strictly increasing sequence of natural numbers such that for each k,
\begin{equation}\label{Eq:Cond_nk}
n_k\ge \frac{k}{1-t_k}.
\end{equation}
For convenience, Let $n_0=0$.

For each $\omega\in \{0,1\}^{<\mathbb N}$, define $\phi_\omega$ as follows. Let $k$ be the unique number such that $n_k\le |\omega|<n_{k+1}$. If $1^k\preceq\omega$, let
\begin{equation}
\phi_\omega=\psi_{i(\omega),\frac12},
\end{equation}
and otherwise, let
\begin{equation}
\phi_\omega=\psi_{i(\omega),\frac13}.
\end{equation}

We now prove the following theorems \ref{Eg:Dim} and \ref{Eg:Zn} which describe that there exists a TIFS $\Psi=\tifs$ for which \eqref{Eq:Cand} does not hold.

\begin{thm}\label{Eg:Dim}
Let $\Psi=\tifs$ be as in above. Then, $\Psi$ is a CTIFS, and $\dim_H(J)=\beta^*(\Psi)=\frac{\log2}{\log3}$.
\end{thm}
\begin{proof}
It is easy to check $\tifs$ is a CTIFS.

We first show that
\begin{equation}
Z_{\infty}^*\left(\frac{\log2}{\log3}\right)\ge 1.
\end{equation}

Take $n$ arbitrarily, and let $A$ be a maximal antichain in $T_{\succ\emptyset}^{(\le n)}$. Then,
\begin{equation}
\sum_{\omega\in A}\|\Phi_\omega'\|^{\frac{\log2}{\log3}}\ge \sum_{\omega\in A}\left(\left(\frac13\right)^{|\omega|}\right)^{\frac{\log2}{\log3}}=\sum_{\omega\in A}\left(\frac12\right)^{|\omega|}.
\end{equation}
Moreover, we have that for each $\omega\in A$,
\begin{equation}
\sum_{\tau\in [\omega]|_n}\frac1{2^n}=2^{n-|\omega|}\cdot \frac1{2^n}=\frac1{2^{|\omega|}}.
\end{equation}
Hence, Lemma \ref{Lem:Cover} implies that
\begin{equation}
\sum_{\omega\in A}\left(\frac12\right)^{|\omega|}=\sum_{\omega\in A}\sum_{\tau\in [\omega]|_n}\frac{1}{2^n}=\sum_{\tau\in T^{(n)}}\frac1{2^n}=1.
\end{equation}
Therefore,
\begin{equation}
Z_n^*\left(\frac{\log2}{\log3}\right)\ge 1.
\end{equation}
Letting $n\to\infty$, we obtain that
\begin{equation}
Z_\infty^*\left(\frac{\log2}{\log3}\right)\ge 1.
\end{equation}

We next suppose that $t>\frac{\log2}{\log3}$. Then,
\begin{equation}
\lim_{n\to\infty}2^n\cdot\left(\frac13\right)^{nt}=0.
\end{equation}

Fix $\varepsilon>0$ arbitrarily, and let $k$ be large enough so that
\begin{equation}\label{Eg:Large1}
\left(\frac1{2^k}\right)^t<\frac\varepsilon2.
\end{equation}
Furthermore, let $n$ be large enough so that
\begin{equation}\label{Eg:Large2}
2^n\cdot\left(\frac13\right)^{nt}\cdot 2^{n_k}<\frac\varepsilon2.
\end{equation}

Let $A=\{1^k\}\cup\{\omega\in T^{(n+n_k)}\mid 1^k\not\prec \omega\}$. It is easy to see that $A$ is a maximal antichain.

Suppose that $\omega\in T^{(n+n_k)}$ and $1^k\not\prec\omega$. If $l\ge n_k$, then $\|\phi_{\omega|_l}'\|=\frac13$. Therefore,
\begin{equation}
\|\Phi_\omega'\|\le \left(\frac13\right)^{|\omega|-n_k}=\left(\frac13\right)^n.
\end{equation}
Combining with \eqref{Eg:Large1} and \eqref{Eg:Large2},
\begin{equation}
\sum_{\omega\in A}\|\Phi_\omega'\|^t\le \frac\varepsilon2+\sum_{\omega\in T^{(n+n_k)}}\left(\frac13\right)^{nt}=\frac\varepsilon2+2^{n+n_k}\cdot\left(\frac13\right)^{nt}<\varepsilon.
\end{equation}

Therefore,
\begin{equation}
Z_\infty^*(t)\le Z_{n+n_k}^*(t)\le \varepsilon
\end{equation}
and it follows that $Z_\infty^*(t)=0$.

Using (ii) of Theorem \ref{Thm:Main}, it is concluded that $\dim_H(J)=\beta^*(\Psi)=\frac{\log2}{\log3}$.
\end{proof}

\begin{thm}\label{Eg:Zn}
Let $\Psi=\tifs$ be as above. Then, for all $t<1$, there exists $N$ such that for all $n\ge N$, $Z_n(t)\ge1$. Here, $Z_n$ is defined as \eqref{Eq:Def_Zn}. Especially, $\liminf_{n\to\infty}Z_n(t)\ge 1$ for all $t<1$, and 
\begin{equation}
\dim_H(J)=\beta^*(\Psi)=\frac{\log2}{\log3}<1=\beta(\Psi).
\end{equation}
\end{thm}
\begin{proof}
Take $t\in[0,1)$ arbitrarily. Let $k_0$ be large enough so that $t_{k_0}\ge t$. Let $N=n_{k_0}$. Take $n\ge N$ and let $k$ be the unique number satisfying $n_k\le n< n_{k+1}$. Since $n\ge n_{k_0}$, we have that $k\ge k_0$. We obtain that

\begin{equation}
Z_n(t)=\sum_{\omega\in T^{(n)}}\|\Phi_\omega'\|^t\ge \sum_{\omega\in[1^k]|_n}\|\Phi_\omega'\|^t.
\end{equation}

Let $\omega\in [1^k]|_n$. Let $l\le n$ and take $k'$ satisfying $n_{k'}\le l<n_{k'+1}$. Then, since $l\le n$, we have $k'\le k$. Since $k'\le l$,
\begin{equation}
(\omega|_l)|_{k'}=\omega|_{k'}=(\omega|_k)|_{k'}=(1^k)|_{k'}=1^{k'}.
\end{equation}
Therefore, $\omega|_l\succeq1^{k'}$ and $\|\phi_{\omega|_l}'\|=\frac12$. Thus, 
\begin{equation}
\|\Phi_\omega'\|^t=\left(\frac12\right)^{nt}.
\end{equation}

Since $t\le t_{k_0}\le t_k$, using \eqref{Eq:Cond_nk} we obtain that
\begin{equation}
Z_n(t)\ge \sum_{\omega\in[1^k]|_n}\left(\frac12\right)^{nt}\ge 2^{n-k}\cdot \left(\frac12\right)^{nt_k}=2^{n(1-t_k)-k}\ge 2^{n_k(1-t_k)-k}\ge 1.
\end{equation}

Therefore, $\liminf_{n\to\infty}Z_n(t)\ge1$. On the other hand, let $t>1$. Then, we have
\begin{equation}
Z_n(t)=\sum_{\omega\in T^{(n)}}\|\Phi_\omega'\|^t\le \sum_{\omega\in T^{(n)}}\left(\frac12\right)^{nt}=2^n\cdot\left(\frac12\right)^{nt}\to0
\end{equation}
as $n\to\infty$. Thus, we have $\beta(\Psi)\le 1$. Therefore, $\beta(\Psi)=1$.

Together with Theorem \ref{Eg:Dim}, we have proved our theorem.
\end{proof}

 \qed  
 
 \ 
 
\noindent {\bf Acknowledgements.} 
The author thanks Professor Hiroki Sumi for valuable comments.

\bibliography{bibTIFS}

@article {MR3449231,
    AUTHOR = {Rempe-Gillen, Lasse and Urba\'nski, Mariusz},
     TITLE = {Non-autonomous conformal iterated function systems and
              {M}oran-set constructions},
   JOURNAL = {Trans. Amer. Math. Soc.},
  FJOURNAL = {Transactions of the American Mathematical Society},
    VOLUME = {368},
      YEAR = {2016},
    NUMBER = {3},
     PAGES = {1979--2017},
      ISSN = {0002-9947,1088-6850},
   MRCLASS = {28A80 (37C45 37F10)},
  MRNUMBER = {3449231},
MRREVIEWER = {Manuel\ Mor\'an},
       DOI = {10.1090/tran/6490},
       URL = {https://doi-org.kyoto-u.idm.oclc.org/10.1090/tran/6490},
}

@book {MR1102677,
    AUTHOR = {Falconer, Kenneth},
     TITLE = {Fractal geometry},
      NOTE = {Mathematical foundations and applications},
 PUBLISHER = {John Wiley \& Sons, Ltd., Chichester},
      YEAR = {1990},
     PAGES = {xxii+288},
      ISBN = {0-471-92287-0},
   MRCLASS = {28A80 (00A69 11K55 28-01 58F13 60G18)},
  MRNUMBER = {1102677},
MRREVIEWER = {Christoph\ Bandt},
}

@book {MR2003772,
    AUTHOR = {Mauldin, R. Daniel and Urba\'nski, Mariusz},
     TITLE = {Graph directed {M}arkov systems},
    SERIES = {Cambridge Tracts in Mathematics},
    VOLUME = {148},
      NOTE = {Geometry and dynamics of limit sets},
 PUBLISHER = {Cambridge University Press, Cambridge},
      YEAR = {2003},
     PAGES = {xii+281},
      ISBN = {0-521-82538-5},
   MRCLASS = {37C45 (28A75 28A78 28A80 37B10 37C30 37F30 82B05)},
  MRNUMBER = {2003772},
MRREVIEWER = {Marc\ Kesseb\"ohmer},
       DOI = {10.1017/CBO9780511543050},
       URL = {https://doi-org.kyoto-u.idm.oclc.org/10.1017/CBO9780511543050},
}

@article {MR2431670,
    AUTHOR = {Barnsley, Michael F. and Hutchinson, John E. and Stenflo,
              \"Orjan},
     TITLE = {{$V$}-variable fractals: fractals with partial self
              similarity},
   JOURNAL = {Adv. Math.},
  FJOURNAL = {Advances in Mathematics},
    VOLUME = {218},
      YEAR = {2008},
    NUMBER = {6},
     PAGES = {2051--2088},
      ISSN = {0001-8708,1090-2082},
   MRCLASS = {28A80 (37C45 60D05)},
  MRNUMBER = {2431670},
MRREVIEWER = {Artemi\ Berlinkov},
       DOI = {10.1016/j.aim.2008.04.011},
       URL = {https://doi-org.kyoto-u.idm.oclc.org/10.1016/j.aim.2008.04.011},
}
\bibliographystyle{plain}
\bigskip

Division of Mathematical and Information Sciences

Graduate School of Human and Environmental Studies

Kyoto University

Yoshida-nihonmatsu-cho, Sakyo-ku, Kyoto, 606-8501, Japan

E-mail: ono.hiromichi.58c@st.kyoto-u.ac.jp
\end{document}